\documentclass[a4paper,12pt,onecolumn]{article}
\usepackage{etex}
\usepackage[vmargin=2cm,hmargin=2cm,headheight=14.5pt,top=2cm,headsep=.5cm]{geometry}

\usepackage{bm}
\usepackage{empheq}
\usepackage{stackrel}
\usepackage{cases}
\usepackage{mathtools}
\usepackage{amsthm,amsmath,amscd}
\usepackage{amssymb,amsfonts}
\usepackage{makeidx}
\usepackage[all]{xy}
\usepackage{perpage}
\usepackage{tikz-cd}
\tikzcdset{every label/.append style = {font = \small}}
\usepackage[symbol]{footmisc}

\MakePerPage[2]{footnote}
\usepackage{graphicx}
\DeclareMathSizes{12}{12}{8}{6}

\usepackage{cite}
\usepackage{url}
\usepackage[charter]{mathdesign}
\usepackage{accents}

\newtheoremstyle{ptheorem}{1em}{0em}{\itshape}{}{\bfseries}{.}{.5em}{\thmname{#1}\thmnumber{ #2}\thmnote{ (\hspace{-.01pt}{#3})}}

\theoremstyle{ptheorem}

\newtheorem{thm}{Theorem}[section]

\newtheorem{lem}[thm]{Lemma}
\newtheorem{cor}[thm]{Corollary}

\newtheoremstyle{hdef}{1em}{0em}{}{}{\bfseries}{.}{.5em}{\thmname{#1}\thmnumber{ #2}\thmnote{ (\hspace{-.01pt}{#3})}}
\theoremstyle{hdef}

\newtheorem{dfn}[thm]{Definition}
\newtheorem{rem}[thm]{Remark}

\makeatletter
\newtheoremstyle{premark}{1em}{0em}{
\addtolength{\@totalleftmargin}{1.5em}
\addtolength{\linewidth}{-1.5em}
\parshape 1 1.5em \linewidth}{}{\scshape}{.}{.5em}{}
\makeatother

\theoremstyle{premark}

\newtheorem{exa}[thm]{Example}

\numberwithin{equation}{section}
\numberwithin{figure}{section}

\DeclareMathOperator{\sign}{sign}

\DeclareMathOperator{\im}{im}

\DeclareMathOperator{\supp}{supp}

\DeclareMathOperator{\Id}{Id}

\DeclareMathOperator{\dif}{d}

\DeclareMathOperator{\BVloc}{BV_{loc}}
\DeclareMathOperator{\AC}{AC}
\DeclareMathOperator{\ACloc}{AC_{loc}}
\DeclareMathOperator{\Loloc}{L_{loc}^1}

\DeclareMathOperator{\Lt}{L^2}
\DeclareMathOperator{\Lo}{L^1}
\DeclareMathOperator{\Dom}{Dom}

\newcommand{\cA}{{\mathcal A}}

\newcommand{\cC}{{\mathcal C}}

\newcommand{\cL}{{\mathcal L}}

\newcommand{\bC}{{\mathbb C}}
\newcommand{\bD}{{\mathbb D}}

\newcommand{\bN}{{\mathbb N}}

\newcommand{\bR}{{\mathbb R}}

\newcommand{\sH}{{\mathsf H}}

\renewcommand{\a}{\alpha}
\renewcommand{\b}{\beta}

\newcommand{\e}{\epsilon}

\renewcommand{\phi}{\varphi}

\newcommand{\<}{\langle}
\renewcommand{\>}{\rangle}

\newcommand{\ol}{\overline}

\newcommand{\n}{{n\in\bN}}

\renewcommand{\<}{\left<}
\renewcommand{\>}{\right>}

\renewcommand{\(}{\left(}
\renewcommand{\)}{\right)}
\newcommand{\olb}[1]{%
  \vbox{\offinterlineskip\ialign{\hfil##\hfil\cr $\rotatebox[origin=c]{90}{$]$}$\cr\noalign{\kern-.45ex}{$#1$}\cr}}}
\allowdisplaybreaks

\begin{document}
\title{{Green's functions for reducible functional differential equations}\footnote{Partially supported by FEDER and Ministerio de Educaci\'on y Ciencia, Spain, project MTM2010-15314}}

\author{
Alberto Cabada \, and F. Adri\'an F. Tojo\footnote{Supported by  FPU scholarship, Ministerio de Educaci\'on, Cultura y Deporte, Spain.} \\
\normalsize
Departamento de An\'alise Ma\-te\-m\'a\-ti\-ca, Facultade de Matem\'aticas,\\ 
\normalsize Universidade de Santiago de Com\-pos\-te\-la, Spain.\\ 
\normalsize e-mail: alberto.cabada@usc.es, fernandoadrian.fernandez@usc.es}
\date{}

\maketitle

\begin{abstract}
In this work we study differential problems in which the reflection operator and the Hilbert transform are involved. We reduce these problems to ODEs in order to solve them. Also, we describe a general method for obtaining the Green's function of reducible functional differential equations and illustrate it with the case of homogeneous boundary value problems with reflection and several specific examples.
\end{abstract}

\section{Introduction}
Some special kind of functional differential equations, called reducible differential equations, can be solved by making operations on them which lead to a related problem with an ODE or system of ODEs (see, for instance, \cite{Sha,Wie2}).\par
{To be more specific, if $\bR[D]$ is the ring of polynomials on the usual differential operator $D$ and $\cA$ is any operator algebra containing $\bR[D]$, then an equation $L x=0$, where $L\in\cA$, is said to be a \textbf{reducible differential equation} if there exits $R\in\cA$ such that $RL\in\bR[D]$. A similar definition could be done for non-constant or complex coefficients.}

This ODE problem can be solved and, of the solutions obtained for it, some may be solutions of the original problem as well. This approach has recently been extended to the obtaining of Green's functions for some of those problems \cite{Toj,Toj2,Toj3,Cab4}.\par
It is important to point out that these transformations necessary to reduce the problem to an ordinary one are of a purely algebraic nature. It is in this sense similar to the \textit{algebraic analysis} theory which, through the study of Ore algebras and modules, obtains important information about some functional problems, including explicit solutions \cite{Clu,Bro}. Nevertheless, the algebraic structures we deal with here are somewhat different, e.\,g., they are not in general Ore algebras. {We refer the reader to \cite{Kor,Reg1,Reg2,Ros} for an algebraic approach to the abstract theory of boundary value problems and its applications to symbolic computation.}\par
Among the reducible functional differential equations, those with reflection have gathered great interest, some of it due to their applications to supersymmetric quantum mechanics \cite{Post, Roy, Gam} or to other areas of analysis like topological methods \cite{Cab5}.\par
In this work we put special emphasis in two operators appearing in the equations: the reflection operator and the Hilbert transform. Both of them have exceptional algebraic properties which make them fit for our approach.\par
In the next section we study the case of operators with reflection and the algebra generated by them, illustrating its properties. In Section 3, we show how we can compute the Green's function of a problem with reflections in a fairly general setting using the properties studied in Section 2. Section 3 introduces a particular case of a more general context. This new setting is studied in Section 4, where we outline the theory for abstract linear operators and prove, as a particular case, the result used in Section 3 to derive the Green's function. Finally, in Section 5, we show that the application of our results extends beyond equations with reflection, study the case of differential equations in which the Hilbert transform is involved and give an example of how to compute the solutions of these equations. Also, we show how these kind of operators relate to the complex polynomials and outline an analogous theory for hyperbolic polynomials.
\section{Differential operators with reflection}
In this Section we will study a particular family of operators, those that are combinations of the differential operator $D$, the pullback operator of the reflection $\phi(t)=-t$, denoted by $\phi^*(f)(t)=f(-t)$, and the identity operator, $\Id$. In order to freely apply the operator $D$ without worrying too much about it's domain of definition, we will consider that $D$ acts on the set of functions locally of bounded variation on $\bR$, $\BVloc(\bR)$\footnote{Since we will be working with $\bR$ as a domain throughout this article, it will be in our interest to take the local versions of the classical function spaces. By local version we mean that, if we restrict the function to a compact set, the restriction belongs to the classical space defined with that compact set as domain for its functions.}.\par
It is well known that any function locally of bounded variation $f\in\BVloc(\bR)$ can be expressed as
\begin{displaymath}f(x)=f(x_0)+\int_{x_0}^xg(y)\dif y+h(x),\end{displaymath}
for any $x_0\in\bR$ where $g\in L^1(\bR)$, and $h$ is the function of derivative zero almost everywhere \cite{Kol}. This implies that the distributional derivative (we will call it \textit{weak derivative} as shorthand) of $f$ is
\begin{displaymath}f'=g+\mu_s,\end{displaymath}
where $\mu_s$ is a singular measure with respect to the Lebesgue measure. In this way, we will define $D\,f:=g$ (we will restate this definition in a more general way further on).\par
We now consider the real abelian group $\bR[D,\phi^* ]$ of generators $\{D^k,\phi^*D^k\}_{k=0}^\infty$ where the powers $D^k$ are taken in the sense of composition (as notation, $D^0=\Id$, and a constant, say $a$, will be considered, in so far as an operator, acting on a function $f$ and returning the product $a\,f$). If we take the usual composition for operators in $\bR[D,\phi^* ]$, we observe that $D\phi^*=-\phi^*D$, so composition is closed in $\bR[D,\phi^* ]$, which makes it a \textit{non commutative algebra}. In general,
$D^k\phi^*=(-1)^k\phi^*D^k$ for $k=0,1,\dots$\par
The elements of $\bR[D,\phi^* ]$ are of the form
\begin{equation}\label{Lop}L=\sum_ia_i\phi^*D^i+\sum_jb_jD^j\in\bR[D,\phi^* ].\end{equation}
For convenience, we consider the sums on $i$ and $j$ such that $i,j\in\{0,1,\dots\}$, but taking into account that the coefficients $a_i,b_j$ are zero for big enough indices.\par
Despite the non commutativity of the composition in $\bR[D,\phi^* ]$ there are interesting relations in this algebra.\par
First notice that $\bR[D,\phi^*]$ is not a unique factorization domain. Take a polynomial $P=D^2+\b D+\a$ where $\a,\b\in \bR$, and define the following operators.\par
If $\b^2-4\a\ge 0$,
\begin{align*}
L_1 & :=D+\frac{1}{2} \left(\beta -\sqrt{\beta ^2-4 \alpha }\right),\\
R_1& :=D+\frac{1}{2} \left(\beta +\sqrt{\beta ^2-4 \alpha }\right),\\
L_2 & :=\phi^*D-\sqrt{2}D+\frac{1}{2} \left(\beta -\sqrt{\beta ^2-4 \alpha }\right) \phi^*+\frac{\left(-\beta +\sqrt{\beta ^2-4 \alpha }\right)}{\sqrt{2}},\\
R_2& :=\phi^*D-\sqrt{2}D-\frac{1}{2} \left(\beta +\sqrt{\beta ^2-4 \alpha }\right) \phi^*-\frac{\left(\beta +\sqrt{\beta ^2-4 \alpha }\right)}{\sqrt{2}},\\
L_3 & :=\phi^*D-\sqrt{2}D+\frac{1}{2} \left(\beta +\sqrt{\beta ^2-4 \alpha }\right) \phi^*-\frac{\left(\beta +\sqrt{\beta ^2-4 \alpha }\right)}{\sqrt{2}},\\
R_3& :=\phi^*D-\sqrt{2}D+\frac{1}{2} \left(-\beta +\sqrt{\beta ^2-4 \alpha }\right) \phi^*+\frac{\left(-\beta +\sqrt{\beta ^2-4 \alpha }\right)}{\sqrt{2}},\\
L_4 & :=\phi^*D+\sqrt{2}D+\frac{1}{2} \left(\beta -\sqrt{\beta ^2-4 \alpha }\right) \phi^*+\frac{\left(\beta -\sqrt{\beta ^2-4 \alpha }\right)}{\sqrt{2}},\\
R_4 & :=\phi^*D+\sqrt{2}D-\frac{1}{2} \left(\beta +\sqrt{\beta ^2-4 \alpha }\right) \phi^*+\frac{\left(\beta +\sqrt{\beta ^2-4 \alpha }\right)}{\sqrt{2}},\\
L_5 & :=\phi^*D+\sqrt{2}D+\frac{1}{2} \left(\beta +\sqrt{\beta ^2-4 \alpha }\right) \phi^*+\frac{\left(\beta +\sqrt{\beta ^2-4 \alpha }\right) }{\sqrt{2}},\\
R_5 & :=\phi^*D+\sqrt{2}D+\frac{1}{2} \left(-\beta +\sqrt{\beta ^2-4 \alpha }\right) \phi^*+\frac{\left(\beta -\sqrt{\beta ^2-4 \alpha }\right)}{\sqrt{2}}.
\end{align*}
If $\b=0$ and $\a\le0$,
\begin{align*}
L_6 & :=\phi^*D+\sqrt{-\a}\phi^*,\\
L_7 & :=\phi^*D-\sqrt{-\a}\phi^*,\\
\end{align*}
If  $\b=0$ and $\a\ge0$,
\begin{align*}
L_8 & :=D+\sqrt{\a}\phi^*,\\
L_9 & :=D-\sqrt{\a}\phi^*,\\
\end{align*}
If  $\b=0$ and $\a\le 1$,
\begin{align*}
L_{10} & :=\phi^*D-\sqrt{1-\a}\phi^*+1,\\
R_{10} & :=-\phi^*D+\sqrt{1-\a}\phi^*+1,\\
L_{11} & :=\phi^*D+\sqrt{1-\a}\phi^*+1,\\
R_{11} & :=-\phi^*D-\sqrt{1-\a}\phi^*+1,\\
\end{align*}
If  $\b=0$, $\a\ne0$ and $\a\le 1$,
\begin{align*}
L_{12} & :=\phi^*D-\sqrt{1-\a}D+\a,\\
R_{12} & :=-\frac{1}{\a}\phi^*D+\frac{\sqrt{1-\a}}{\a}D+1,\\
L_{13} & :=\phi^*D+\sqrt{1-\a}D+\a,\\
R_{13} & :=-\frac{1}{\a}\phi^*D-\frac{\sqrt{1-\a}}{\a}D+1,\\
\end{align*}
Then,
\begin{displaymath}P=L_1R_1=R_1L_1=R_2L_2=R_3L_3=R_4L_4=R_5L_5,\end{displaymath}
and, when $\b=0$,
\begin{displaymath}P=-L_6^2=-L_7^2=L_8^2=L_9^2=R_{10}L_{10}=L_{10}R_{10}=R_{11}L_{11}=L_{11}R_{11}=R_{12}L_{12}=L_{12}R_{12}=R_{13}L_{13}=L_{13}R_{13}.\end{displaymath}
Observe that only $L_1$ and $R_1$ commute in the case of $\b\ne0$.\par
This rises the question on whether we can decompose every differential polynomial $P$ in the composition of two `order one' elements of $\bR[D,\phi^* ]$, but this is not the case in general. Just take $Q=D^2+D+1$ (observe that $Q$ is not in any of the aforementioned cases). Consider a decomposition of the kind
\begin{displaymath}(a\phi^*D+bD+c\phi^*+d)(e\phi^*D+gD+h\phi^*+j)=Q,\end{displaymath}
where $a,b,c,d,e,g,h$ and $j$ are real coefficients to be determined. The resulting system
\begin{displaymath}\left\{\begin{aligned}
d h + c j  & = 0,\\ d e - c g + b h + a j  & = 0 ,\\
b e - a g & = 0,\\ -a e + b g & = 1,\\
c h + d j & = 1,\\ -c e + d g + a h + b j  & = 1,
\end{aligned}\right.\end{displaymath}
has no solution for real coefficients.\par
Let $\bR[D]$ be the ring of polynomials with real coefficients on the variable $D$. The following result states a very useful property of the algebra $\bR[D,\phi^* ]$.
\begin{thm}\label{thmdec}
Take $L$ as defined in \eqref{Lop} and take
\begin{equation}\label{Rop}R=\sum_{k}a_k\phi^*D^k+\sum_{l}(-1)^{l+1}b_lD^l\in\bR[D,\phi^* ].\end{equation}  Then $RL=LR\in\bR[D]$.
\end{thm}
\begin{proof}
\begin{equation}\label{RL}\begin{aligned}RL= &\sum_{i,k}(-1)^ka_ia_kD^{i+k}+\sum_{j,k}b_ja_k\phi^*D^{j+k}+\sum_{i,l}(-1)^l(-1)^{l+1}a_ib_l\phi^*D^{i+l}+\sum_{j,l}(-1)^{l+1}b_jb_lD^{j+l}\\ = &\sum_{i,k}(-1)^ka_ia_kD^{i+k}+\sum_{j,l}(-1)^{l+1}b_jb_lD^{j+l}.\end{aligned}\end{equation}
Hence, $RL\in\bR[D]$.\par
Observe that, if we take $R$ in the place of $L$ in the hypothesis of the Theorem, we obtain $L$ in the place of $R$ and so, by expression \eqref{RL} $LR\in\bR[D]$.
\end{proof}
\begin{rem} Some interesting remarks on the coefficients of the operator $S=RL$ defined in Theorem \ref{thmdec} can be made.\par
If we have
\begin{displaymath}S=\sum_k c_kD^k=RL=\sum_{i,k}(-1)^ka_ia_kD^{i+k}+\sum_{j,l}(-1)^{l+1}b_jb_lD^{j+l},\end{displaymath}
then
\begin{displaymath}c_k=\sum_{i=0}^k(-1)^i(a_ia_{k-i}-b_ib_{k-i}).\end{displaymath}
A closer inspection reveals that
\begin{displaymath}c_k=\begin{dcases} 0, & k \text{ odd,} \\
2\sum_{i=0}^{\frac{k}{2}-1}\(-1\)^i\(a_ia_{k-i}-b_ib_{k-i}\)+\(-1\)^\frac{k}{2}\(a_\frac{k}{2}^2-b_\frac{k}{2}^2\) & k \text{ even.}\end{dcases}.\end{displaymath}
This has some important consequences. If $L=\sum_{i=0}^n a_i\phi^*D^i+\sum_{j=0}^n b_jD^j$ with $a_n\ne0$ or $b_n\ne0$, we have that $c_{2n}=(-1)^n(a_n^2-b_n^2)$\footnote{This is so because if $i\in\{0,\dots,n-1\}$, then $2n-i\in \{n+1,\dots,2n\}$ and $a_n$ (respectively $b_n$) are non-zero only for $n\in\{0,\dots,n\}$.} and so if $a_n=\pm b_n$ then $c_{2n}=0$. This shows that composing two elements of $\bR[D,\phi^* ]$ we can get another element which has simpler terms in the sense of derivatives of less order. We illustrate this with two examples.\par
Take $n\ge3$, $L=\phi^*D^n+D^n+D-\Id $ and $R=\phi^*D^n-(-1 )^nD^n+D+\Id $. Then
$RL=2D^{\alpha(n)}+D^2-\operatorname{Id}$ where $\alpha(n)=n$ if $n$ is even and $\alpha(n)=n+1$ if $n$ is odd.
If we take $n\ge0 $, $L=\phi^*D^{2n+1}+D^{2n+1}+\Id $ and $R=\phi^*D^n-(-1 )^nD^n+\Id $ then $RL=-\Id$.
\end{rem}
\begin{exa}\label{firstexa} Consider the equation
\begin{displaymath}x^{(3)}(t)+x^{(3)}(-t)+x(t)=\sin t.\end{displaymath}
Applying the operator $\phi^*D^3+D^3-\Id$ to both sides of the equation we obtain $x(t)=\sin t+2\cos t$. This is the unique solution of the equation, to which we had not imposed any extra conditions.\end{exa}

\section{Boundary Value Problems}
In this section we obtain the Green's function of boundary value problems with reflection and constant coefficients. We point out that the same approach used in this section is also valid for initial problems among other types of conditions.\par
Let $I=[a,b]\subset\bR$ be an interval and $f\in\Lo(I)$. Consider now the following problem with the usual derivative.
\begin{equation}\begin{aligned}\label{lccbvp}Su(t):= & \sum_{k=0}^na_ku^{(k)}(t)=f(t),\ t\in I,\\ B_iu:= & \sum_{j=0}^{n-1}\a_{ij}u^{(j)}(a)+\b_{ij}u^{(j)}(b)=0,\ i=1,\dots,n.
\end{aligned}\end{equation}
The following Theorem from \cite{Cab6} states the cases where we can find a unique solution for problem \eqref{lccbvp}\footnote{ In \cite{Cab6}, this result is actually stated for nonconstant coefficients, but the case of constant coefficients is enough for our purposes.}.
\begin{thm}\label{thmgf}
Assume the following homogeneous problem has a unique solution
\begin{equation*}\label{lccbvp2}Su(t)=0,\ t\in I,\ B_iu=0,\ i=1,\dots n.
\end{equation*}
Then there exists a unique function, called \textbf{Green's function}, such that
\begin{itemize}
\item[(G1)] $G$ is defined on the square $I^2$.
\item[(G2)] The partial derivatives $\frac{\partial^kG}{\partial t^k}$ exist and are continuous on $I^2$ for $k=0,\dots,n-2$.
\item[(G3)] $\frac{\partial^{n-1}G}{\partial t^{n-1}}$ and $\frac{\partial^nG}{\partial t^n}$ exist and are continuous on $I^2\backslash\{(t,t)\ :\ t\in I\}$.
\item[(G4)] The lateral limits $\frac{\partial^{n-1}G}{\partial t^{n-1}}(t,t^+)$ and $\frac{\partial^{n-1}G}{\partial t^{n-1}}(t,t^-)$ exist for every $t\in(a,b)$ and
\begin{displaymath}\frac{\partial^{n-1}G}{\partial t^{n-1}}(t,t^-)-\frac{\partial^{n-1}G}{\partial t^{n-1}}(t,t^+)=\frac{1}{a_n}.\end{displaymath}
\item[(G5)] For each $s\in(a,b)$ the function $G(\cdot,s)$ is a solution of the differential equation $Su=0$ on $I\backslash\{s\}$.
\item[(G6)] For each $s\in(a,b)$ the function $G(\cdot,s)$ satisfies the boundary conditions $B_iu=0\ i=1,\dots,n$.
\end{itemize}
Furthemore, the function $u(t):=\int_a^bG(t,s)f(s)\dif s$ is the unique solution of the problem \eqref{lccbvp}.
\end{thm}
Using the properties (G1)--(G6) and Theorem \ref{thmdec} one can prove Theorem \ref{thmdei}. The proof of this result will be a direct consequence of Theorem  \ref{thmmostgen}.\par
Given an operator $\cL$ for functions of one variable, define the operator $\cL_\vdash$ as $\cL_\vdash G(t,s):=\cL(G(\cdot,s))|_{t}$ for every $s$ and any suitable function $G$.
\begin{thm}\label{thmdei} Let $I=[-T,T]$. Consider the problem
\begin{equation}\label{rbvp}Lu(t)=h(t),\ t\in I,\ B_iu=0,\ i=1,\dots,n,
\end{equation}
where $L$ is defined as in \eqref{Lop}, $h\in L^1(I)$ and
\begin{displaymath}B_iu:=\sum_{j=0}^{n-1}\a_{ij}u^{(j)}(-T)+\b_{ij}u^{(j)}(T).\end{displaymath}
Then, there exists $R\in \bR[D,\phi^* ]$ (as in \eqref{Rop}) such that $S:=RL\in\bR[D]$ and the unique solution of problem \eqref{rbvp} is given by $\int_a^bR_\vdash G(t,s)h(s)\dif s$ where $G$ is the Green's function associated to the problem $Su=0$, $B_iRu=0$, $B_iu=0$, $i=1,\dots,n$, assuming that it has a unique solution.
\end{thm}
{ For the following example, let us explain some notations. Let $k,p\in\bN$. We denote by $W^{k,p}(I)$ the Sobolev Space defined by
\begin{displaymath}W^{k,p}(I) = \left \{ u \in L^p(I) : D^{\alpha}u \in L^p(I) \,\, \forall \alpha \leq k \right \}.\end{displaymath}
Given a constant $a\in\bR$ we can consider the pullback by this constant as a functional $a^*:\cC(I)\to\bR$ such that $a^*f=f(a)$ in the same way we defined it for functions.}
\begin{exa}\label{exagf}
Consider the following problem.
\begin{equation}\label{prooc}u''(t)+a\,u(-t)+b\,u(t)=h(t),\ t\in I,\quad  u(-T)=u(T),\ u'(-T)=u'(T).
\end{equation}
 where $h\in W^{2,1}(I)$. Then, the operator we are considering is $L=D^2+a\,\phi^*+b$. If we take $R:=D^2-a\,\phi^*+b$, we have that $RL=D^4+2b\,D^2+b^2-a^2$.\par
The boundary conditions are $((T^*)-(-T)^*)u=0$ and $((T^*)-(-T)^*)Du=0$. Taking this into account,  we add the conditions \begin{displaymath}0=((T^*)-(-T)^*)Ru=((T^*)-(-T)^*)(D^2-a\,\phi^*+b)u=((T^*)-(-T)^*)D^2u,\end{displaymath}
\begin{displaymath}0=((T^*)-(-T)^*)RDu=((T^*)-(-T)^*)(D^2-a\,\phi^*+b)Du=((T^*)-(-T)^*)D^3u.\end{displaymath}
That is, our new \textit{reduced} problem is
\begin{equation}\label{proocr}u^{(4)}(t)+2b\,u''(t)+(b^2-a^2)u(t)=f(t),\ t\in I,\quad  u^{(k)}(-T)=u^{(k)}(T),\ k=0,\dots,3.
\end{equation}
where $f(t)=R\,h(t)=h''(t)-a\,h(-t)+b\,h(t)$.\par
Observe that this problem is equivalent to the system of equations (a chain of order two problems)
\begin{align*}
u''(t)+(b+a)u(t) & =v(t),\ t\in I,\quad u(-T)=u(T),\ u'(-T)=u'(T),\\
v''(t)+(b-a)v(t) & =f(t),\ t\in I,\quad v(-T)=v(T),\ v'(-T)=v'(T).
\end{align*}
Thus, it is clear that
\begin{displaymath}u(t)=\int_{-T}^TG_1(t,s)v(s)\dif s,\ v(t)=\int_{-T}^TG_2(t,s)f(s)\dif s,\end{displaymath}
where, $G_1$ and $G_2$ are the Green's functions related to the previous second order problems. Explicitly, in the case $b>|a|$ (the study for other cases would be analogous),
\begin{equation*}
2\sqrt{b+a}\sin(\sqrt{b+a}\,T)G_1(t,s)=\begin{cases} \cos \sqrt{b+a}(T+s-t) & \text{if}\quad s\le t,\\\cos \sqrt{b+a}(T-s+t) & \text{if}\quad s>t.\end{cases}
\end{equation*}
and
\begin{equation*}
2\sqrt{b-a}\sin(\sqrt{b-a}\,T)G_2(t,s)=\begin{cases} \cos \sqrt{b-a}(T+s-t) & \text{if}\quad s\le t,\\\cos \sqrt{b-a}(T-s+t) & \text{if}\quad s>t.\end{cases}
\end{equation*}

Hence, the Green's function $G$ for problem \eqref{proocr} is given by
\begin{displaymath}G(t,s)=\int_{-T}^TG_1(t,r)G_2(r,s)\dif r.\end{displaymath}
Therefore, using Theorem \ref{thmdei}, the Green's function for problem \eqref{prooc} is
\begin{displaymath}\ol G(t,s)=R_\vdash G(t,s)=\frac{\partial^2 G}{\partial t^2}(t,s)-a\,G(-t,s)+b\,G(t,s).\end{displaymath}
\end{exa}
\begin{rem} We can reduce the assumptions on the regularity of $h$  to $h\in\Lo(I)$ just taking into account the density of $W^{2,1}(I)$ in $L^1(I)$.
\end{rem}
\begin{rem}\label{remord}
Example \ref{firstexa} illustrates the importance of the existence and uniqueness of solution of the problem $Su=0,\ B_iRu=0,\ B_iu=0$ in the hypothesis of Theorem \ref{thmdei}. In general, when we compose two linear ODEs, respectively of orders $m$ and $n$ and a number $m$ and $n$ of conditions, we obtain a new problem of order $m+n$ and $m+n$ conditions. As we see this is not the case in the reduction of  Theorem \ref{thmdei}. In the case the order of the reduced problem is less than $2n$ anything is possible: we may have an infinite number of solutions, no solution or uniqueness of solution being the problem non-homogeneous. The following example illustrates this last case.
\end{rem}
\begin{exa}
Consider the problem
\begin{equation}\label{sppro} Lu(t):= u^{(4)}(t)+u^{(4)}(-t)+u''(-t)=h(t),\ t\in[-1,1]\quad u(1)=u(-1)=0,\end{equation}
where $h\in W^{4,1}([-1,1])$.\par
For this case, $Ru(t):=-u^{(4)}(t)+u^{(4)}(-t)+u''(-t)$ and the reduced equation is $RLu=2u^{(6)}+u^{(4)}=Rh$, which has order $6<2\cdot4=8$, so there is a reduction of the order. Now we have to be careful with the new \textit{reduced} boundary conditions.
\begin{equation}\begin{aligned}\label{redbc} B_1u(t) & =u(1)=0,\\
B_2u(t) & =u(-1)=0,\\
 B_1Ru(t) &=-u^{(4)}(1)+u^{(4)}(-1)+u''(-1)=0,\\
 B_2Ru(t) & =-u^{(4)}(-1)+u^{(4)}(-1)+u''(1)=0,\\
 B_1Lu(t) &=u^{(4)}(1)+u^{(4)}(-1)+u''(-1)=h(1),\\
 B_2Lu(t) & =u^{(4)}(-1)+u^{(4)}(-1)+u''(1)=h(-1).\\
\end{aligned} \end{equation}
 Being the two last conditions the obtained from applying the original boundary conditions to the original equation.\par
 \eqref{redbc} is a system of linear equations which can be solved for $u$ and its derivatives as
 \begin{equation}
 u(1)=u(-1)=0,\ u''(1)=-u''(-1)=\frac{1}{2}(h(1)-h(-1)),\ u^{(4)}(\pm 1)=\frac{h(\pm 1)}{2}.
 \end{equation}
Consider now the reduced problem
 \begin{align*}
&  2u^{(6)}(t)+u^{(4)}(t)=Rh(t)=:f(t),\ t\in[-1,1], \\
&  u(1)=u(-1)=0,\ u''(1)=-u''(-1)=\frac{1}{2}(h(1)-h(-1)),\ u^{(4)}(\pm 1)=\frac{h(\pm 1)}{2},
 \end{align*}
 and the change of variables $v(t):=u^{(4)}(t)$.
 Now we look the solution of
\begin{displaymath}  2v''(t)+v(t)=f(t),\ t\in[-1,1],\ v(\pm 1)=\frac{h(\pm 1)}{2}, \end{displaymath}
Which is given by
\begin{displaymath}v(t)=\int_{-1}^1G(t,s)f(s)\dif s-\frac{h(1)\csc\sqrt{2}}{2}\sin\(\frac{t-1}{\sqrt 2}\)+\frac{h(-1)\csc\sqrt{2}}{2}\sin\(\frac{t+1}{\sqrt 2}\),\end{displaymath}
where
\begin{displaymath}G(t,s):=\frac{\csc\sqrt{2}}{\sqrt{2}}\begin{dcases}
\sin \left(\frac{s+1}{\sqrt{2}}\right) \sin \left(\frac{t-1}{\sqrt{2}}\right), & -1\le s\le t\le1, \\
 \sin \left(\frac{s-1}{\sqrt{2}}\right) \sin \left(\frac{t+1}{\sqrt{2}}\right), & -1\le t<s\le1.
\end{dcases}\end{displaymath}
Now, it is left to solve the problem
\begin{displaymath}u^{(4)}(t)=v(t),\ u(1)=u(-1)=0,\ u''(1)=-u''(-1)=\frac{1}{2}(h(1)-h(-1)).\end{displaymath}
\end{exa}
The solution is given by
\begin{displaymath}u(t)=\int_{-1}^{1}K(t,s)v(s)\dif s+\frac{h(1)-h(-1)}{12}t(t - 1) (t + 1).\end{displaymath}
where
\begin{displaymath}K(t,s)=\frac{1}{12}\begin{dcases}
  (s+1) (t-1) \left(s^2+2 s+t^2-2 t-2\right), & -1\le s\le t\le 1, \\
  (s-1) (t+1) \left(s^2-2 s+t^2+2 t-2\right), &  -1\le t<s\le 1.
\end{dcases}\end{displaymath}
Hence, taking $J(t,s)=\int_{-1}^1H(t,r)G(r,s)\dif s$,
\begin{align*} & J(t,s):=  \\ &\frac{\csc\sqrt 2}{12\sqrt 2}\begin{dcases}
 \sqrt{2}\sin(\sqrt{2}) (s+1) (t-1)  [s (s+2)+(t-2) t-14]+24 \cos \left(\frac{s-t+2}{\sqrt{2}}\right)-24 \cos \left(\frac{s+t}{\sqrt{2}}\right), \\-1\le s\le t\le 1, \\
 \sqrt{2}\sin(\sqrt{2}) (s-1) (t+1)[(s-2) s+t (t+2)-14]+24 \cos \left(\frac{s-t-2}{\sqrt{2}}\right)-24 \cos \left(\frac{s+t}{\sqrt{2}}\right), \\ -1\le t<s\le 1.
\end{dcases}
\end{align*}
Therefore,
\begin{align*}u(t) & =\int_{-1}^{1}J(t,s)f(s)\dif s-\frac{h(1)\csc\sqrt{2}}{2}\left[\frac{1}{6} (t-5) (t-1) (t+3) \sin \left(\sqrt{2}\right)+4 \sin \left(\frac{t-1}{\sqrt{2}}\right)\right] \\ &+\frac{h(-1)\csc\sqrt{2}}{2}\left[\frac{1}{6} (t-3) (t+1) (t+5) \sin \left(\sqrt{2}\right)+4 \sin \left(\frac{t+1}{\sqrt{2}}\right)\right] +\frac{h(1)-h(-1)}{12}t(t - 1) (t + 1).
\end{align*}
\section{The reduced problem}
The usefulness of a theorem of the kind of Theorem \ref{thmdei} is clear, for it allows the obtaining of the Green's function of any problem of differential equations with constant coefficients and involutions, generalizing the works \cite{Cab4, Toj, Toj2}.  The proof of this Theorem relies heavily on the properties $(G1)-(G6)$, so our main goal now is to consider abstractly these properties in order to apply them in a more general context with different kinds of operators.\par
Let $X$ be a vector subspace of $\Loloc(\bR)$, and $(\bR,\tau)$ the real line with its usual topology. Define $X_U:=\{f|_U\ :\ f\in X\}$ for every $U\in\tau$ (observe that $X_U$ is a vector space as well). Assume that $X$ satisfies the following property.\par
\textbf{(P)} \textit{For every partition of $\bR$, $\{S_j\}_{j\in J}\cup\{N\}$, consisting of measurable sets where $N$ has no accumulation points and the $S_j$ are open, if $f_j\in X_{S_j}$ for every $j\in J$, then there exists $f\in X$ such that $f|_{S_j}=f_j$ for every $j\in J$.}\par
\begin{exa}
The set of locally absolutely continuous functions $\ACloc(\bR)\subset\Loloc(\bR)$ does not satisfy \textbf{(P)}. To see this just take the following partition of $\bR$: $S_1=(-\infty,0)$, $S_2=(0,+\infty)$, $N=\{0\}$ and consider $f_1\equiv 0$, $f_2\equiv 1$. $f_j\in \AC(\bR)_{S_j}$ for $j=1,2$, but any function $f$ such that $f|_{S_j}=f_j$, $j=1,2$ has a discontinuity at $0$, so it cannot be absolutely continuous. That is, \textbf{(P)} is not satisfied. 
\end{exa}

\begin{exa}
$X=\BVloc(\bR)$ satisfies \textbf{(P)}. Take a partition of $\bR$, $\{S_j\}_{j\in J}\cup\{N\}$, consisting of measurable sets where $N$ has no accumulation points and the $S_j$ are open and a family of functions $(f_j)_{j\in J}$ such that $f_j\in X_{S_j}$ for every $j\in J$. We can further assume, without lost of generality, that the $S_j$ are connected. Define a function $f$ such that $f|_{S_j}:=f_j$ and $f|_N=0$. Take a compact set $K\subset\bR$. Then, by Bozano-Weierstrass' and Heine-Borel's Theorems, $K\cap N$ is finite for $N$ has no accumulation points. Therefore, $J_K:=\{j\in J : S_j\cap K\ne\emptyset\}$ is finite as well. To see this denote by $\partial S$ the boundary of a set $S$ and observe that $N\cup K=\cup_{j\in J}\partial (S_j\cap K)$ and that the sets $\partial (S_j\cap K)\cap \partial (S_k\cap K)$ are finite for every $j,k\in J$.\par
Thus, the variation of $f$ in $K$ is $V_K(f)\le\sum_{j\in J_K}V_{S_j}(f)<\infty$ since $f$ is of bounded variation on each $S_j$. Hence, $X$ satisfies \textbf{(P)}.
\end{exa}
Throughout this section we will consider a function space $X$ satisfying \textbf{(P)} and two families of linear operators $L=\{L_U\}_{U\in\tau}$ and $R=\{R_U\}_{U\in\tau}$ that satisfy
\begin{description}
\item[\textit{Locality:}]$L_U\in\cL(X_U,\Loloc(U))$, $R_U\in\cL(\im(L_U),\Loloc(U))$,
\item[\textit{Restriction:}] $L_V(f|_V)=L_U(f)|_V$, $R_V(f|_V)=R_U(f)|_V$ for every $U,V\in\tau$ such that $V\subset U$\footnote{The definitions here presented of $L$ and $R$ are deeply related to Sheaf Theory. Since the authors want to make this work as self-contained as possible, we will not deepen into that fact.}.
\end{description}
The following definition allows us to give an example of an space that satisfies the properties of locality and restriction.
\begin{dfn} Let $f:\bR\to\bR$ and assume there exists a partition $\{S_j\}_{j\in J}\cup\{N\}$ of $\bR$ consisting of measurable sets where $N$ is of zero Lebesgue measure such that satisfying that the weak derivative $g_i$ exists for every $f|_{S_j}$, then a function $g$ such that $g|_{S_j}=g_j$ is called the \textbf{very weak derivative} (vw-derivative) of $f$.
\end{dfn}
\begin{rem} The vw-derivative is uniquely defined save for a zero measure set and is equivalent to the weak derivative for absolutely continuous functions.\par
Nevertheless, the vw-derivative is different from the derivative of distributions. For instance, the derivative of the Heavyside function in the distributional sense is de Dirac delta at 0, whereas its vw-derivative is zero. What is more, the kernel of the vw-derivative is the set of functions which are constant on a family of open sets $\{S_j\}_{j\in J}$ such $\bR\backslash(\cup_{j\in J} S_j)$ has Lebesgue measure zero.
\end{rem}
\begin{exa}
Take $X=\BVloc(\bR)$ and $L=D$ to be the very weak derivative. Then $L$ satisfies the locality and restriction hypotheses.
\end{exa}
\begin{rem}
The vw-derivative, as defined here, is the $D$ operator defined in Section 2 for functions of bounded variation. In other words, the vw-derivative ignores the jumps and considers only those parts with enough regularity.
\end{rem}
\begin{rem}
The locality property allows us to treat the maps $L$ and $R$ as if they were just linear operators in $\cL(X,\Loloc(\bR))$ and $\cL(\im(X),\Loloc(\bR))$ respectively, although we must not forget their more complex structure.
\end{rem}
Assume $X_U\subset\im( L_U)\subset\im (R_U)$ for every $U\in\tau$. $B_i\in\cL(\im( R_\bR),\bR)$, $i=1,\dots,m$ and $h\in\im(L_\bR)$. Consider now the following problem
\begin{equation}\label{prineq}
Lu=h,\  B_iu=0,\ i=1,\dots,m.
\end{equation}
Let
\begin{displaymath}Z:=\{G:\bR^2\to\bR\ :\ G(t,\cdot)\in X\cap \Lt(\bR)\text{ and }\supp\{G(t,\cdot)\}\text{ is compact},\  s\in\bR\}.\end{displaymath}
 $Z$ is a vector space.\par

Let $f\in\im(L_\bR)$ and consider the problem
\begin{equation}\label{prineq2}
RLv=f,\  B_iv=0,\  B_iRv=0,\ i=1,\dots,m.
\end{equation}
Let $G\in Z$ and define the operator $H_G$ such that $H_G(h)|_t:=\int_\bR G(t,s)h(s)\dif s$.
We have now the following theorem relating problems \eqref{prineq} and \eqref{prineq2}.
\begin{thm}\label{thmmostgen}
Assume $L$ and $R$ are the aforementioned operators with the locality and restriction properties and let $h\in\Dom(R_\bR)$. Assume $L$ commutes with $R$ and that there exists $G\in Z$ such that\par
$\begin{array}{rl}
(I) & (RL)_\vdash G=0,\\
(II) & B_{i\,\vdash} G=0,\ i=1,\dots,m,\\
(III) & ( B_iR)_\vdash G=0,\ i=1,\dots,m,\\
(IV) & RLH_Gh=H_{(RL)_\vdash G}h+h,\\
(V) & LH_{R_\vdash G}h=H_{L_\vdash R_\vdash G}h+h.\\
(VI) &  B_iH_G= H_{B_{i\,\vdash} G},\ i=1,\dots,m,\\
(VII) &  B_iRH_G= B_iH_{R_\vdash G}=H_{( B_iR)_\vdash G},\ i=1,\dots,m,\\
\end{array}$\par
Then, $v:=H_G(h)$ is a solution of problem \eqref{prineq2} and $u:=H_{R_\vdash G}(h)$ is a solution of problem \eqref{prineq}.
\end{thm}
\begin{proof}
$(I)$ and $(IV)$ imply that  \begin{displaymath}RLv=RLH_Gh=H_{(RL)_\vdash G}h+h=H_0h+h=h.\end{displaymath} On the other hand, $(III)$ and $(VII)$ imply that, for every $i=1,\dots,m$,
\begin{displaymath} B_iRv= B_iRH_Gh=H_{( B_iR)_\vdash G}h=0.\end{displaymath}
All the same, by $(II)$ and $(VI)$,
\begin{displaymath} B_iv= B_iH_Gh=H_{B_{i\,\vdash} G}h=0.\end{displaymath}
Therefore, $v$ is a solution to problem \eqref{prineq2}.\par
Now, using $(I)$ and $(V)$ and the fact that $LR=RL$, we have that
\begin{displaymath}Lu=LH_{R_\vdash G}h=H_{L_\vdash R_\vdash G}h+h=H_{(LR)_\vdash G}h+h=H_{(RL)_\vdash G}h+h=h.\end{displaymath}

Taking into account $(III)$ and $(VII)$,
\begin{displaymath} B_iu= B_iH_{R_\vdash G}(h)=H_{( B_iR)_\vdash G}h=0,\ i=1,\dots,m.\end{displaymath}
Hence, $u$ is a solution of problem \eqref{prineq}.
\end{proof}
The following Corollary is proved in the same way as the previous Theorem.
\begin{cor} Assume $G\in Z$ satisfies\par
$\begin{array}{rl}
(1) & L_\vdash G=0,\\
(2) &  B_{i\,\vdash} G=0,\ i=1,\dots,m,\\
(3) & LH_Gh=H_{L_\vdash G}h+h,\\
(4) &  B_iH_Gh=H_{ B_{i\,\vdash} G}h.
\end{array}$\par
Then $u=H_Gh$ is a solution of problem \eqref{prineq}.
\end{cor}
\begin{proof}[Proof of Theorem \ref{thmdei}] Originally we would need to take $h\in\Dom(R)$, but by a simple density argument ($\mathcal{C}^\infty(I)$ is dense in $\Lo(I)$) we can take $h\in \Lo(I)$. If we prove that the hypothesis of Theorem \ref{thmmostgen} are satisfied, then the existence of solution will be proved. First, Theorem \ref{thmdec} guarantees the commutativity of $L$ and $R$. Now, Theorem \ref{thmgf} implies hypothesis $(I)-(VII)$ of Theorem \ref{thmmostgen} in terms of the vw-derivative.\par
Indeed, $(I)$ is straightforward from $(G5)$. $(II)$ and $(III)$ are satisfied because $(G1)-(G6)$ hold and $B_iu=B_iRu=0$. and $(G4)$ imply $(IV)$ and $(V)$. $(VI)$ and $(VII)$ hold because of $(G2)$, $(G5)$ and the fact that the boundary conditions commute with the integral.\par
On the other hand, the solution to problem \eqref{rbvp} must be unique for, otherwise, the reduced problem $Su=0$, $B_iRu=0$, $B_iu=0$, $i=1,\dots,n$ would have several solutions, contradicting the hypotheses.
\end{proof}
\par The following Lemma, in the line of \cite{Toj2}, extends the application of Theorem \ref{thmdei} to the case of non-constant coefficients with some restrictions for problems similar to the one in Example \ref{exagf}.
\begin{lem}
\label{lemgen} Consider the problem
\begin{equation}\label{eqW2}u''(t)+a(t)\,u(-t)+b(t)\,u(t)=h(t), u(-T)=u(T),\end{equation}
where $a\in W^{2,1}_{\operatorname{loc}}(\mathbb{R})$ is nonnegative and even,
\begin{displaymath}b=k\,a+\frac{a''}{4\,a}-\frac{5}{16}\left(\frac{a'}{a}\right)^2,
\end{displaymath}
 for some constant $k\in\mathbb{R}$, $k^2\ne 1$ and $b$ is integrable.\par
 Define $A(t):=  \int_0^t\sqrt{a(s)}\operatorname{d} s$, consider
  \begin{displaymath}u''(t)+u(-t)+k\,u(t)=h(t),\ u(-A(T))=u(A(T))\end{displaymath}
 and assume it has a Green's function $G$.

 Then \begin{displaymath}u(t)=\int_{-T}^{T}H(t,s)h(s)\operatorname{d} s\end{displaymath} is a solution of problem \eqref{eqW2} where
\begin{align*}H(t,s):= & \sqrt[4]{\frac{a(s)}{a(t)}}G(A(t),A(s))
\end{align*}
and $H(t,\cdot)h(\cdot)$ is assumed to be integrable in $[-T,T]$.
\end{lem}
\begin{proof} Let $G$ be the Green's function of the problem \begin{displaymath}u''(t)+u(-t)+c\,u(t)=h(t),\ u(-A(T))=u(A(T)),\ u\in \operatorname{W^{2,1}_{loc}}(\bR).\end{displaymath}
Now, we show that $H$ satisfies the equation, that is,
\begin{displaymath}\frac{\partial^2 H}{\partial t^2}(t,s)+a(t)H(-t,s)+b(t)H(t,s)=0\text{ for a.\,e. } t,s\in\bR.\end{displaymath}
\begin{align*} & \frac{\partial^2 H}{\partial t^2}(t,s)=\frac{\partial^2 }{\partial t^2}\left[\sqrt[4]{\frac{a(s)}{a(t)}}G(A(t),A(s))\right]=\frac{\partial }{\partial t}\left[-\frac{a'(t)}{4}\sqrt[4]{\frac{a(s)}{a^5(t)}}G(A(t),A(s))+\sqrt[4]{a(s)a(t)}\frac{\partial G }{\partial t}(A(t),A(s))\right]\\ =& -\frac{a''(t)}{4}\sqrt[4]{\frac{a(s)}{a^5(t)}}G(A(t),A(s))+\frac{5}{16}(a'(t))^2\sqrt[4]{\frac{a(s)}{a^9(t)}}G(A(t),A(s))+ \sqrt[4]{a(s)a^3(t)}\frac{\partial^2 G }{\partial t^2}(A(t),A(s)).
\end{align*}
Therefore,
\begin{align*} & \frac{\partial^2 H}{\partial t^2}(t,s)+a(t)H(-t,s)+b(t)H(t,s) \\ = &
\sqrt[4]{a(s)a^3(t)}\frac{\partial^2 G }{\partial t^2}(A(t),A(s))
+a(t)\sqrt[4]{\frac{a(s)}{a(t)}}G(-A(t),A(s))+c\,a(t)\sqrt[4]{\frac{a(s)}{a(t)}}G(A(t),A(s))\\ = &
\sqrt[4]{a(s)a^3(t)}\(\frac{\partial^2 G }{\partial t^2}(A(t),A(s))
+G(-A(t),A(s))+c\,G(A(t),A(s))\)=0.
\end{align*}
The boundary conditions are satisfied as well.
\end{proof}
\begin{rem} The same construction of Lemma \ref{lemgen} is valid for the case of the initial value problem. We illustrate this in the following example.
\end{rem}
\begin{exa}
Let $a(t)=|t|^p$, $k>1$. Taking $b$ as in Lemma \ref{lemgen},
\begin{displaymath}b(t)=k|t|^p-\frac{p(p+4)}{16t^2},\end{displaymath}
consider problems
\begin{equation}\label{prole1}u''(t)+a(t)\,u(-t)+b(t)\,u(t)=h(t),\ u(0)=u'(0)=0\end{equation}
and 
\begin{equation}\label{prole2}u''(t)+u(-t)+k\,u(t)=h(t),\ u(0)=u'(0)=0.\end{equation}
Using an argument similar as the one in Example \ref{exagf} and considering $R=D^2-\varphi^*+k$, we can reduce problem \eqref{prole2} to
\begin{equation}\label{prole3}u^{(4)}(t)+2ku''(t)+(k^2-1)u(t)=f(t),\ u^{(j)}(0)=0,\ j=0,\dots,3,\end{equation}
which can be decomposed in
\begin{align*}
u''(t)+(k+1)u(t) & =v(t),\ t\in I,\quad u(0)=u'(0)=0,\\
v''(t)+(k-1)v(t) & =f(t),\ t\in I,\quad v(0)=v'(0)=0,
\end{align*}
which have as Green's functions, respectively,
\begin{displaymath}\tilde G_1(t,s)=\frac{\sin\left(\sqrt{k+1}\, (t-s)\right)}{\sqrt{k+1}}\chi_0^t(s),\ t\in\mathbb{R},\end{displaymath}
\begin{displaymath}\tilde G_2(t,s)=\frac{\sin\left(\sqrt{k-1}\, (t-s)\right)}{{\sqrt{k-1}}}\chi_0^t(s),\ t\in\mathbb{R},\end{displaymath}
where $\chi_0^t(s)=1$ if $s\in[0,t]$ and $\chi_0^t(s)=-1$ if $s\in[-t,0)$. Then, the Green's function for problem \eqref{prole3} is
\begin{align*}  G(t,s)  = & \int_{s}^t\tilde G_1(t,r)\tilde G_2(r,s)\operatorname{d} r \\ = & \frac{1}{2\sqrt{k^2-1}} \left[\sqrt{k-1} \sin\left(\sqrt{k+1} (s-t)\right)-\sqrt{k+1} \sin\left(\sqrt{k-1} (s-t)\right)\right]\chi_0^t(s),\end{align*}
Observe that
\[R_\vdash G(t,s)=-\left[\frac{\sin \left(\sqrt{k-1} (s-t)\right)}{2 \sqrt{k-1}}+\frac{\sin \left(\sqrt{k+1} (s-t)\right)}{2 \sqrt{k+1}}\right]\chi_0^t(s).\]
Hence, considering
\begin{displaymath}A(t):=\frac{2}{p+2}|t|^\frac{p}{2}t,\end{displaymath}
the Green's function of  problem \eqref{prole1} follows the expression
\begin{displaymath}H(t,s):=  \sqrt[4]{\frac{a(s)}{a(t)}}G(A(t),A(s)),\end{displaymath}
This is,
\begin{align*}H(t,s)= & -\left|\frac{s}{t}\right|^\frac{p}{4}\left[\frac{\sin \left(\frac{2 \sqrt{k-1} \left(s \left| s\right| ^{p/2}-t \left| t\right| ^{p/2}\right)}{p+2}\right)}{2 \sqrt{k-1}}+\frac{\sin \left(\frac{2 \sqrt{k+1} \left(s \left| s\right| ^{p/2}-t \left| t\right| ^{p/2}\right)}{p+2}\right)}{2 \sqrt{k+1}}\right]\chi_0^t(s).\end{align*}
\end{exa}
\section{The Hilbert transform and other algebras}
In this section we devote our attention to new algebras to which we can apply the previous results. To achieve this goal we recall the definition and remarkable properties of the Hilbert transform (see \cite{King}).\par
We define the \textbf{Hilbert transform} $\mathsf{H}$ of a function $f$ as
\begin{displaymath}\mathsf{H} f(t):=\frac{1}{\pi}\lim_{\e\to\infty}\int_{-\e}^\e\frac{f(s)}{t-s}\dif s\equiv\frac{1}{\pi}\int_{-\infty}^\infty\frac{f(s)}{t-s}\dif s,\end{displaymath}
where the last integral is to be understood as the Cauchy principal value.\par
Among its properties, we would like to point out the following.
\begin{itemize}
\item $\sH:L^p(\bR)\to L^p(\bR)$ is a linear bounded operator for every $p\in (1,+\infty)$ and
\begin{displaymath}\|\sH\|_p=\begin{cases} \tan\frac{\pi}{2 p}, & p\in(1,2], \\ \cot\frac{\pi}{2 p}, & p\in[2,+\infty),\end{cases}\end{displaymath}
in particular $\|\sH\|_2=1$.
\item  \textit{$\mathsf{H}$ is an anti-involution:} $\sH^2=-\Id$.
\item Let $\sigma(t)=at+b$ for $a,b\in\bR$. Then $\sH\sigma^*=\sign(a)\sigma^*\sH$ (in particular, $\sH\phi^*=-\phi^*\sH$). Furthermore, if a linear bounded operator $\mathsf{O}:L^p(\bR)\to L^p(\bR)$ satisfies this property, $\mathsf{O}=\beta\sf H$ where $\b\in\bR$.
\item \textit{$\sH$ commutes with the derivative:} $\sH D=D\sH$.
\item $\sH(f*g)=f*\sH g=\sH*g$ where $*$ denotes the convolution.
\item \textit{$\sH$ is an isometry in $L^2(\bR)$:} $\<\sH f,\sH g\>=\<f,g\>$ where $\<\ ,\ \>$ is the scalar product in $L^2(\bR)$. In particular $\|\sH f\|_2=\|f\|_2$.
\end{itemize}
Consider now the same construction we did for $\bR[D,\phi^*]$ changing $\phi^*$ by $\sH$ and denote this algebra as $\bR[D,\sH]$. In this case we are dealing with a commutative algebra. Actually, this algebra is isomorphic to the complex polynomials $\bC[D]$. Just consider the isomorphism
\begin{displaymath}\begin{CD}\bR[D,\sH] @>\displaystyle\Xi>> \bC[D]\\
\sum\limits_{j}(a_j\sH+b_j)D^j @>>> \sum\limits_{j}(a_j\, i+b_j)D^j.\end{CD}\end{displaymath}
Observe that $\Xi|_{\bR[D]}=\Id|_{\bR[D]}$.

We now state a result analogous to Theorem \ref{thmdec}.
\begin{thm} \label{thmdec2}
Take
\begin{equation*}\label{Lop2}L=\sum\limits_{j}(a_j\sH+b_j)D^j\in\bR[D,\sH]\end{equation*}
and define
\begin{equation*}\label{Rop2}R=\sum\limits_{j}(a_j\sH-b_j)D^j.\end{equation*}
Then $LR=RL\in\bR[D]$.
\end{thm}
\begin{rem} Theorem \ref{thmdec2} is clear from the point of view of $\bC[D]$. Since $\Xi(R)=-\ol{\Xi(L)}$,
\begin{displaymath}RL=\Xi^{-1}(-\Xi(L)\ol{\Xi(L)})=\Xi^{-1}(-|\Xi(L)|^2).\end{displaymath}
Therefore, $|\Xi(L)|^2\in\bR[D]$, implies $RL\in\bR[D]$.
\end{rem}
\begin{rem} Since $\bR[D,\sH]$ is isomorphic to $\bC[D]$, the Fundamental Theorem of Algebra also applies to  $\bR[D,\sH]$, which shows a clear classification of the decompositions of an element of  $\bR[D,\sH]$ in contrast with those of  $\bR[D,\phi^*]$ which, in Section 2, was shown not to be a unique factorization domain.
\end{rem}
In the following example we will use some properties of the Hilbert transform:
\begin{align*}\sH\cos & =\sin, \\ \sH\sin & =-\cos, \\ \sH( tf(t)){(t)} & =t\,\sH f(t)-\frac{1}{\pi}\int_{-\infty}^\infty f(s)\dif s,
\end{align*}
where the integral is considered as the principal value.
\begin{exa}
Consider the problem
\begin{equation}\label{proh} Lu(t)\equiv u'(t)+a\sH u(t)=h(t):=\sin a t,\ u(0)=0,
\end{equation}
where $a>0$. Composing the operator $L=D+a\sH$ with the operator $R=D-a\sH$ we obtain $S=RL=D^2+a^2$, the harmonic oscillator operator. The extra boundary conditions obtained applying $R$ are $u'(0)-a\sH u(0)=0$. The general solution to the problem $u''(t)+a^2u(t)=Rh(t)=2a\cos at, u(0)=0$ is given by
\begin{displaymath}v(t)=\int_0^t\frac{\sin\(a\, [t-s]\)}{a}Rh(s)\dif s+\a\sin at=(t+\a)\sin at,\end{displaymath}
where $\a$ is a real constant. Hence,
\begin{displaymath}\sH v(t)=-(t+\a)\cos at.\end{displaymath}
If we impose the boundary conditions $u'(0)-a\sH u(0)=0$ then we get $\a=0$. Hence, the unique solution of problem \eqref{proh} is
 \begin{displaymath}u(t)=t\sin at.\end{displaymath}
\end{exa}
\begin{rem} It is easy to check that the kernel of $D+a\sH$ ($a>0$) is spanned by $\sin t$ and $\cos t$ and therefore, the kernel of $D-a\sH$ is just $0$. This defies, in the line of Remark \ref{remord}, the usual relation between the degree of the operator and the dimension of the kernel which is held for ODEs, that is, the operator of a linear ODE of order $n$ has a kernel of dimension $n$. In this case we have the order 1 operator $D+a\sH$ with a dimension 2 kernel and the injective order 1 operator $D-a\sH$.
\end{rem}
Now, we consider operators with reflection and Hilbert transforms, and denote the algebra as $\bR[D,\sH,\phi^*]$. We can again state a reduction Theorem.
\begin{thm}\label{thmdec3}
Take
\begin{equation*}\label{Lop3}L=\sum_{i}a_i\phi^*\sH D^i+\sum_{i}b_i\sH D^i + \sum_{i}c_i\phi^* D^i+\sum_{i}d_i D^i\in\bR[D,\sH,\phi^*]\end{equation*}
and define
\begin{equation*}\label{Rop3}R=\sum_{j}a_j\phi^*\sH D^j+\sum_{j}(-1)^jb_j\sH D^j + \sum_{j}c_j\phi^* D^j-\sum_{j}(-1)^jd_j D^j.\end{equation*}
Then $LR=RL\in\bR[D]$.
\end{thm}
\subsection{Hyperbolic numbers as operators}
Finally, we use the same idea behind the isomorphism $\Xi$ to construct an operator algebra isomorphic to the algebra of polynomials on the hyperbolic numbers.\par
The hyperbolic numbers\footnote{See \cite{Ant, Toj4} for an introduction to hyperbolic numbers and some of their properties and applications.} are defined, in a similar way to the complex numbers, as follows,
\begin{displaymath}\mathbb D=\{x+jy\ :\ x,y\in\mathbb R,\ j\not\in\mathbb R,\ j^2=1\}.\end{displaymath}
The arithmetic in $\mathbb D$ is that obtained assuming the commutative, associative and distributive properties for the sum and product. In a parallel fashion to the complex numbers, if $w\in\mathbb D$, with $w=x+jy$, we can define
\begin{displaymath}\overline{w}: =x-jy,\quad \Re(w):=x,\quad \Im(w):= y,\end{displaymath}
and, since $w\overline w=x^2-y^2\in\mathbb R$, we set
\begin{displaymath}|w|:=\sqrt{|w\overline w|},\end{displaymath}
which is called the \textbf{Minkowski norm}. It is clear that $|w_1w_2|=|w_1||w_2|$ for every $w_1,w_2\in\mathbb D$ and, if $|w|\ne0$, then $w^{-1}=\overline w/|w|^2$. If we add the norm
\begin{displaymath}\|w\|=\sqrt{2(x^2+y^2)},\end{displaymath}
we have that $(\mathbb D, \|\cdot\|)$ is a Banach algebra, so the exponential and the hyperbolic trigonometric functions are well defined. Although, unlike $\mathbb C$, $\mathbb D$ is not a division algebra (not every non-zero element has an inverse), we can derive calculus (differentiation, integration, holomorphic functions\dots) for $\mathbb D$ as well \cite{Ant}.\par
In this setting, we want to derive an operator $J$ defined on a suitable space of functions such that satisfies the same algebraic properties as the hyperbolic imaginary unity $j$. In other words, we want the map
\begin{displaymath}\begin{CD}\bR[D,J] @>\displaystyle\Theta>> \bD[D]\\
\sum\limits_{k}(a_kJ+b_k)D^j @>>> \sum\limits_{k}(a_k\, j+b_k)D^k.\end{CD}\end{displaymath}
to be an algebra isomorphism. This implies:
\begin{itemize}
\item $J$ is a linear operator,
\item $J\not\in\bR[D]$.
\item $J^2=\Id$, that is, $J$ is an involution,
\item $JD=DJ$.
\end{itemize}
There is a simple characterization of linear involutions on a vector space: every linear involution $J$ is of the form
\begin{displaymath}J=\pm(2P-\Id)\end{displaymath}
where $P$ is a projection operator, that is, $P^2=P$. It is clear that $\pm(2P-\Id)$ is, indeed a linear operator and an involution. On the other hand, it is simple to check that, if $J$ Is a linear involution, $P:=(\pm J+\Id)/2$ is a projection, so $J=\pm(2P-\Id)$.\par
Hence, it is sufficient to look for a projection $P$ commuting with de derivative.

\begin{exa} Consider the space $W=\Lt([-\pi,\pi])$ and define
\begin{displaymath}P f(t):=\sum_\n \int_{-\pi}^\pi f(s)\cos(2\,n\,s)\dif s\,\cos(2\,n\,t)\text{ for every} f\in W,\end{displaymath}
that is, take only the sum over the even coefficients of the Fourier series of $f$. Clearly $PD=DP$. $J:=2P-\Id$ satisfies the aforementioned properties.
\end{exa}
The algebra $\bR[D,J]$, being isomorphic to $\bD[D]$, satisfies also very good algebraic properties (see, for instance, \cite{Poo}). In order to get an analogous theorem to Theorem \ref{thmdec} for the algebra $\bR[D,J]$ it is enough to take, as in the case of $\bR[D,J]$, $R=\Theta^{-1}(\ol{\Theta(L)})$.

\end{document}